\documentclass{amsart}
\usepackage{amssymb}
\usepackage{stmaryrd}

\newtheorem{theorem}{Theorem}[section]
\newtheorem{lemma}[theorem]{Lemma}
\newtheorem{cor}[theorem]{Corollary}
\newtheorem{prop}[theorem]{Proposition}
\newtheorem{rem}[theorem]{Remark}

\theoremstyle{definition}

\theoremstyle{remark}

\numberwithin{equation}{section}

\newcommand{\FF}{{\mathbb{F}}}

\newcommand{\bC}{{\mathbf{C}}}
\newcommand{\bZ}{{\mathbf{Z}}}

\newcommand{\bN}{{\mathbf{N}}}

\newcommand{\Stab}{{\operatorname{Stab}}}

\newcommand{\Syl}{{\operatorname{Syl}}}

\newcommand{\Gal}{{\operatorname{Gal}}}

\newcommand{\Ker}{{\operatorname{Ker}}}

\newcommand{\n}{{\mbox{\rm I$\!$N}}}
\newcommand{\GF}{\mbox{GF}}
\newcommand{\X}{{\mbox{$\setminus$}\mbox{$\!\!\!/$}}}
\newcommand{\ug}{\ \raisebox{-.3em}{$\stackrel{\scriptstyle \leq}
{\scriptstyle \sim}$} \ }

\begin{document}

\title{regular and p-regular orbits of solvable linear groups}

\author{Thomas Michael Keller, YONG YANG}
\address{Department of Mathematics,
    Texas State University at San Marcos,
    San Marcos, TX 78666, USA.}
\makeatletter
\email{keller@txstate.edu, yang@txstate.edu}
\makeatother

\thanks{The research of the first author was partially supported by NSA standard grant MSPF08G-206}
\subjclass[2000]{20D10}
\date{}



\begin{abstract}
  Let $V$ be a faithful $G$-module for a finite group $G$ and let $p$ be a prime dividing $|G|$. An orbit $v^G$ for the action of $G$ on $V$ is $p$-regular if $|v^G|_p=|G:\bC_G(v)|_p=|G|_p$. Zhang asks the following question in ~\cite{Zhang}. Assume that a finite solvable group $G$ acts faithfully and irreducibly on a vector space $V$ over a finite field $\FF$. If $G$ has a $p$-regular orbit for every prime $p$ dividing $|G|$, is it true that $G$ will have a regular orbit on $V$? In ~\cite{LuCao}, L\"{u} and Cao construct an example showing that the answer to this question is no, however the example itself is not correct. In this paper, we study Zhang's question in detail. We construct examples showing that the answer to this question is no in general. We also prove the following result. Assume a finite solvable group $G$ of odd order acts faithfully and irreducibly on a vector space $V$ over a field of odd characteristic. If $G$ has a $p$-regular orbit for every prime $p$ dividing $|G|$, then $G$ will have a regular orbit on $V$.

\end{abstract}

\maketitle
\section{Introduction} \label{sec:introduction8}

Let $V$ be a faithful $G$-module for a finite group $G$ and let $p$ be a prime dividing $|G|$. An orbit $v^G$ for the action of $G$ on $V$ is $p$-regular if $|v^G|_p=|G:\bC_G(v)|_p=|G|_p$. Zhang asks the following question in ~\cite{Zhang}. Assume that a finite solvable group $G$ acts faithfully and irreducibly on a vector space $V$ over a finite field $\FF$. If $G$ has a $p$-regular orbit for every prime $p$ dividing $|G|$, is it true that $G$ will have a regular orbit on $V$? In ~\cite{LuCao}, L\"{u} and Cao construct an example showing that the answer to this question is no. However the example itself is not correct. We mention that Lewis is the first to observe that the example in ~\cite{LuCao} is wrong in his review ~\cite{Lewis}. In this paper, we study Zhang's question in detail.

First we construct examples showing that the answer to this question is no in general.
\begin{enumerate}
\item Example $1$: Let $H \cong Z_3$ acts faithfully and irreducibly on $V_1=\FF_2^2$ and $G \cong H \wr Z_5$ acts on $V=\FF_2^{10}$. Clearly $G$ acts faithfully and irreducibly on $V$. The group $G$ has a $3$-regular orbit and a $5$-regular orbit, but it has no regular orbit on $V$.
\item Example $2$: Let $H$ be a solvable group acts faithfully, irreducibly and primitively on $\FF_7^2$ and $H \cong Q_8 \rtimes S_3$, $\bZ(H) \cong Z_2$. We may view $H$ to be a matrix group and define $G$ to be the central product of $H$ with $H$, i.e. $G \cong H \Ydown H$. Thus $G$ acts faithfully, irreducibly and primitively on $V=\FF_7^4$. $|G|=1152=2^7 \cdot 3^2$. By direct calculation of GAP ~\cite{GAP}, the lengths of all the orbits of $G$ on $V$ are given in the following list. \[(1, 48, 48, 48, 144, 144, 144, 192, 192, 192, 288, 288, 288, 384)\] From the list we know that $G$ has a $2$-regular orbit and a $3$-regular orbit, but it has no regular orbit on $V$.
\end{enumerate}

In Example $1$, $G$ is a group of odd order induced from a group of odd order acting on a vector space over a field of characteristic $2$. In Example $2$, $G$ is a solvable group of even order acting on a vector space over a field of odd characteristic. Based on these examples, it is natural to ask the following question. Assume that a finite solvable group $G$ of odd order acts faithfully and irreducibly on a vector space $V$ over a field of odd characteristic. If $G$ has a $p$-regular orbit for every prime $p$ dividing $|G|$, is it true that $G$ will have a regular orbit on $V$? We answer this question affirmatively in the main theorem of this paper. We prove the following.
\begin{theorem}
Assume that a finite solvable group $G$ of odd order acts faithfully and irreducibly on a vector space $V$ over a field of odd characteristic. If $G$ has a $p$-regular orbit for every prime $p$ dividing $|G|$, then $G$ will have a regular orbit on $V$.
\end{theorem}

\section{Notation and Lemmas} \label{sec:Notation and Lemmas}
Before we prove the main result, we extract some important information from the work of A.Turull ~\cite[Section 1]{turull} in the following propositions. Note that Turull's results are stated for a prime $p$ but the same arguments will work for replacing $p$ with a prime power $q$. We include the proof here for completeness.

Notation: Let $q$ be a prime power and $n$ an integer.
\begin{enumerate}
\item $F(q^n)=\GF(q^n)^{\times}$ (the multiplicative group of the field of $q^n$ elements).
\item $\Gal(q^n)=\Gal(\GF(q^n):\GF(q))$.
\item $G(q^n) = \Gal (q^n) \ltimes F(q^n)$.
\item If $\sigma \in \Gal(q^n)$ and $y \in F(q^n)$ we denote $N_{\sigma}(y)=\prod_{\tau \in \langle \sigma \rangle}\tau(y)$.
\item Suppose that $s$ is a prime and divides $n=|\Gal(q^n)|$. Define $GN(q^n,s)=A \ltimes N \subseteq G(q^n)$ where $A$ is a subgroup of $\Gal(q^n)$ of order $s$ and $N=\{x \in F(q^n): \prod_{\sigma \in A} \sigma(x)=1\}$.
\end{enumerate}

\begin{prop}\label{prop1}

a) Let $\sigma \in \Gal (q^n)$ be of order $s$ and set
$N= \{ x \in F(q^n): N_{\sigma}(x) = 1\}$.
Then we have that $x \in N$ iff $x = \frac {\sigma(y)} y$  for some $y \in F(q^n)$. Furthermore we have
\[|N|=\frac {q^n-1}{q^{n/s}-1}\]

b) $N$ is a cyclic normal subgroup of $GN(q^n,s)$ of index $s$ and every element of $GN(q^n,s)$ not in $N$ has order $s$.

c) For any prime $r$ dividing $|N|$, we have $r \geq s$ and either $r=s$ or $r \equiv 1(s)$. If $r$ is a prime dividing $|N|$ with $r\not=s$, and $b\in N$ is of order $r$, then $\langle\sigma\rangle\langle b\rangle$ is a Frobenius group of order $sr$ with Frobenius kernel $\langle b\rangle$.

d) Let $A$ be a subgroup of $G(q^n)$. Set $B=A \cap F(q^n)$. There exists a conjugate $A_1$ in $G(q^n)$ of $A$ such that for any $x \in A \backslash B$ of prime order $q_x$, there is $\sigma \in A_1 \cap \Gal(q^n)$ with $\sigma \neq 1$ such that $\sigma^{q_x}=1$.

\end{prop}
\begin{proof}
a) Let $y \in F(q^n)$ and clearly $N_{\sigma}(\sigma(y)/y)=1$, so that $f(y)=\sigma(y)/y$ is a homomorphism of $F(q^n)$ into $N$.
If for some positive integer $r$ and $z \in N$ we have $z=\sigma^r(y)/y$, then we have
\[z=\frac {\sigma(\sigma^{r-1}(y)\sigma^{r-2}(y)...y)}{\sigma^{r-1}(y)\sigma^{r-2}(y)...y},\]

so that we may replace $\sigma$ by $\sigma^r$ provided that $(r,s)=1$. So we may suppose
\[\sigma(x)=x^{q^{(n/s)}}.\] Then

\[N_{\sigma}(x)=x \cdot x^{q^{(n/s)}}...x^{q^{(n/s)(s-1)}}=x^v \]

where $v=1 + q^{(n/s)}+...+q^{(n/s)(s-1)}=\frac {q^n-1} {q^{(n/s)}-1}$.

Thus $|N|=(\frac {q^n-1} {q^{(n/s)}-1}, q^n-1)=\frac {q^n-1} {q^{(n/s)}-1}$.

Furthermore $f(x)=x^{q^{(n/s)}-1}$, so that $|\Ker f|=(q^{(n/s)}-1, q^n-1)=q^{(n/s)}-1$.

Thus $f$ is surjective and this shows a).

b) Take $\sigma \in \Gal(q^n)$ and $x \in F(q^n)$, then $(\sigma x)^{o(\sigma)}=N_{\sigma}(x)$ and b) is clear.

c) We know that if $x\in Z(GN(q^n,s))$, then $x^s=1$, and so if $r$ is a prime dividing $|N|$ with $r\not=s$, and $b\in N$ is of order $r$, then $\langle\sigma\rangle\langle b\rangle$ is a Frobenius group of order $sr$ with Frobenius kernel $\langle b\rangle$ and in particular $r \equiv 1(s)$.

d) Write $|\Gal(q^n)|=q_1^{\epsilon_1}...q_t^{\epsilon_t}q_{t+1}^{\epsilon_{t+1}}...q_r^{\epsilon_r}$ a product of primes such that $A \backslash B$ has elements of order $q_i$ $i=1,...,t$, but not of order $q_{t+1},...,q_r$. Assume also that $q_1 < q_2<...<q_t$.
Take $A^*$ a conjugate of $A$ which contains elements $\sigma_1,...,\sigma_{i-1} \in \Gal(q^n)$ of order $q_1,...,q_{i-1}$ respectively, with $i$ as large as possible. Assume $i \leq t$.

Take $x \in A^* \backslash B$ of order $q_i$. We may write $x=\sigma_i \alpha$ with $\sigma_i \in \Gal(q^n)$ of order $q_i$ and $\alpha \in F(q^n)$. We have $1 =(\sigma_i \alpha)^{q_i}=N_{\sigma_i}(\alpha)$.

Since $\sigma_j \in A^*$ for $j=1,..., i-1$ we have $[\sigma_j, x] \in A^*$ and $[\sigma_j,x]=[\sigma_j,\sigma_i \alpha]=\sigma_j^{-1} \alpha^{-1} \sigma_i^{-1}\sigma_j\sigma_i \alpha= [\sigma_j, \alpha]$. On the other hand $\sigma_j$ normalizes $\langle \alpha \rangle$, and since $N_{\sigma_i}(\alpha)=1$, any prime dividing $|\langle \alpha \rangle|$ is at least as large as $q_i$ by c). Therefore $q_j \nmid |\langle \alpha \rangle|$ and \[\langle \alpha \rangle = \bC_{\langle \alpha \rangle} (\sigma_j) \times  [\sigma_j, \langle \alpha \rangle]. \]
Since $[\sigma_j,\langle \alpha \rangle]=\langle [\sigma_j, \alpha] \rangle \subseteq A^*$ and $\sigma_i \alpha \in A^*$, there is $y \in \bC_{\langle \alpha \rangle}(\sigma_j)$ such that $\sigma_i y \in A^*$.

This process actually gives $y \in \bC_{\langle \alpha \rangle}(\sigma_1,...,\sigma_{i-1})$ such that $\sigma_i y \in A^*$ and (of
course) $N_{\sigma_i}(y) = 1$.

Now $y \in F(q^{n/(q_1...q_{i-1})})$ and by a) there is $z \in F(q^{n/(q_1...q_{i-1})})$ such that $y= \sigma_i^{-1}(z)/z$.

Take ${A^*}^z$. Then, since $z \in F(q^{n/(q_1...q_{i-1})})$, $\sigma_j \in {A^*}^z$ for $j=1,...,i-1$.

$\sigma_i= \sigma_i \sigma_i^{-1} (z^{-1}) y z = z^{-1} \sigma_i y z \in {A^*}^z$, a contradiction. So we have d).
\end{proof}

\begin{prop}\label{prop2}
$G(q^n)$ acts on the left in a natural way on $F(q^n)$, with $F(q^n)$ acting by multiplication and $\Gal(q^n)$ in its natural way. Let $A \subseteq G(q^n)$ and $B = A \cap F(q^n)$. Then the following are equivalent:

A) $A$ has a regular orbit on $F(q^n)$.

B) For any prime $s$, $GN(q^n, s)$ is not conjugate in $G(q^n)$ to a subgroup of $A$.

C) For any prime $s$ such that $A \backslash B$ has an element of order $s$, if we take $\sigma_s \in \Gal(q^n)$ of order $s$ and $N_s= \{x \in F(q^n): N_{\sigma_s}(x)=1 \}$. Then $N_s \not \subseteq B$.

\end{prop}
\begin{proof}
Set $F=F(q^n)$ and $\pi=\{s$ a prime: $A \backslash B$ has an element of order $s \}$. By Proposition ~\ref{prop1} d), we may assume that, for any $s \in \pi$, $\sigma_s \in \Gal(q^n) \cap A$ is an element of order $s$. Now B) and C) are equivalent.

Clearly we have $\bigcup_{a\in A \backslash \{ 1\}}\bC_F(a)=\bigcup_{a\in {A \backslash B}} \bC_F(a)$.

So
\[\bigcup_{a\in A \backslash \{ 1\}}\bC_F(a)=\bigcup_{s\in\pi}\bigcup_{b\in B}\bC_F(b\cdot\sigma_s)\]

But $\bigcup_{b\in B}\bC_F(b\cdot\sigma_s)=\{x \in F: \frac {\sigma_s(x)}{x} \in B\}$ is a subgroup of $F$ and is proper iff $N_s \not \subseteq B$,
by Proposition ~\ref{prop1} a). Since $F$ is cyclic, $F=\bigcup_{s\in\pi}\bigcup_{b\in B}\bC_F(b\cdot\sigma_s)$ iff $\bigcup_{b\in B}\bC_F(b\cdot\sigma_s)=F$ for some $s \in \pi$, or equivalently $N_s \subseteq B$ for some $s \in \pi$. This shows Proposition ~\ref{prop2}.
\end{proof}

As in \cite{manz/wolf}, if $V$ is a finite vector space of dimension $n$ over $\GF(q)$,
where $q$ is a prime power,
we denote by $\Gamma(q^n)=\Gamma(V)$ the semilinear group of $V$, i.e.,
\[\Gamma(V)=\{x\mapsto
ax^\sigma\ |\ x\in\GF(q^n),a\in\GF(q^n)^\times,\sigma\in\Gal(\GF(q^n):\GF(q))\},\]
and we write
\[\Gamma_0(V)=\{x\mapsto ax\ |\ x\in\GF(q^n),a\in\GF(q^n)^\times\}\]
for the subgroup of multiplications; note that this acts fixed point freely on
$\GF(q^n)$.\\

\begin{cor}\label{cor1}
Let $q$ be a prime power, $n\in\n$ and consider $V=\GF(q^n)$ as an $n$-dimensional vector space over $\GF(q)$, and let $A\leq\Gamma(V)$. Let $B=A\cap\Gamma_0(V)$, $\pi=\{r|r$ is a prime such that $A \backslash B$ has an element of order $r\}$. If $A$ does not have a regular orbit on $V$, then there exists a prime $s \in \pi$ such that for any $v \in V$ there is an element of order $s$ centralizing $v$.

\end{cor}
\begin{proof}
By Proposition ~\ref{prop1} d), we may assume that for $r\in\pi$, $\sigma_r\in A\cap\Gal(\GF(q^n):\GF(q))$ is an element of order $r$. With this it follows from the proof of Proposition ~\ref{prop2} that if $A$ does not have a regular orbit on $V$, then
\[V=\bigcup_{a\in A \backslash \{ 1\}}C_V(a)=\bigcup_{r\in\pi}\bigcup_{b\in B}C_V(b\cdot\sigma_r)=\bigcup_{b\in B}C_V(b\cdot\sigma_{s})\]
for a suitable prime $s\in\pi$. Since for $b\in B$ the order of $b\cdot\sigma_s$ is divisible by $o(\sigma_sB)=s$, we see that every $v\in V$ is fixed by an element of order $s$.
\end{proof}

\begin{lemma}  \label{lem1}
Let $S$ be a transitive solvable permutation group on $\Omega$ with $|\Omega|=m$. If $|S|$ is odd, then S has a regular orbit on the power set $P(\Omega)$ of $\Omega$.
\end{lemma}
\begin{proof}
This is Gluck's Theorem ~\cite[Corollary 5.7]{manz/wolf}.
\end{proof}

\section{Main Theorem} \label{sec:maintheorem}

Now we prove a key result on our way to our main result.



\begin{theorem}\label{kellertheorem}
Let $q$ be a prime power. Let $G$ be a finite group and $V$ a finite, faithful, irreducible $\GF(q)G$-module.
Suppose that $V$ is induced from a submodule $W$ such that for $H=N_G(W)$ we have that $H/C_G(W)$ is isomorphic to an
(irreducible) subgroup of $\Gamma(W)$. Let $N=\bigcap\limits_{g\in G}H^g\unlhd G$. Assume that $|N|$ is odd. Suppose
that $N$ has an $r$-regular orbit on $V$ for every prime number $r$.
Then $N$ has a regular orbit on $V$.
\end{theorem}
\begin{proof}
First observe that by Clifford theory we can write $V_N=V_1\oplus\ldots\oplus V_m$ (for some $m\in\n$) for homogeneous
components $V_i$, and we may assume that $V_1=W$. Moreover $G/N$ is isomorphic to a subgroup of $S_m$, permuting the
$V_i$ ($i=1,\ldots,m$) transitively.\\

We will now prove the following statement
\[(\dagger)\quad \mbox{If }U\leq N,\mbox{ then }U\mbox{ has a regular
orbit on }V.\]
Applying $(\dagger)$ with $U=N$ clearly gives the assertion of the
theorem.\\

We prove ($\dagger$) by contradiction.
So let $U\leq N$ be a minimal counterexample, i.e., every subgroup of $U$ has a regular orbit on $V$, while $U$
does not have such an orbit. Clearly $U>1$. \\
Now as
\[U\ug\Gamma(V_1)\times\ldots\times \Gamma(V_m)\cong\Gamma(V_1)^m\]
(where $\Gamma(V_1)^m$ denotes the direct product of $m$ copies of $\Gamma(V_1)$ and $\ug$ means "is isomorphic to a subgroup of"), throughout the proof we will
identify $U$ with its isomorphic copies in these bigger groups, as
needed.\\

If $U\leq\Gamma_0(V_1)^m$, then clearly $U$ has a regular orbit on $V$
which contradicts the choice of $U$. Thus $U/(U\cap\Gamma_0(V_1)^m)$ is nontrivial, and we let $s$ be the smallest
prime dividing $|U/(U\cap\Gamma_0(V_1)^m)|$. (By the way, we will not use the minimality of $s$ until (2) further
down the proof. Everything up to there is valid for any prime divisor of $|U/(U\cap\Gamma_0(V_1)^m)|$.)\\

As $V$ is a faithful $U$-module there exists an $i$ such that there is an element $x\in U$ of $s$-power order such
that, if we read $U/C_U(V_i)\ug\Gamma(V_i)$, we have that $xC_U(V_i)\not\in\Gamma_0(V_i)$ and $x\not\in C_U(V_i)$;
in particular $s$ divides $|\Gamma(V_i)/\Gamma_0(V_i)|$. Without loss of generality we may assume that $i=1$.
Put $C=C_U(V_1)$. Then $C<U$, so let $C\leq U_0\lhd U$ such that $|U:U_0|=s$. Now as $U_0<U$, we know that $U_0$
has a regular orbit on $V$. Let $w\in V$ be a representative of such an orbit, and write $w=\sum\limits_{i=1}^mv_i$
with $v_i\in V_i$ for $i=1,\ldots,m$. As $1=C_{U_0}(w)=\bigcap\limits_{i=1}^mC_{U_0}(v_i)$, we obviously may assume that
$v_i\not=0$ for $i=1,\ldots,m$. Next put $v=\sum\limits_{i=2}^mv_i$ and observe that $C_C(v)=1$, so $v$ is in a regular
orbit of $C$ on $X:=V_2\oplus\ldots\oplus V_m$.\\

Next we claim that $C_U(v)$ is isomorphic to a subgroup of $U/C$.
To see this, define $\phi:C_U(v)\rightarrow U/C$ by $\phi(u)=uC$ for $u\in C_U(v)$. Clearly $\phi$ is a homomorphism,
and if $u\in\ker(\phi)$, then $u\in C\cap C_U(v)=C_C(v)=1$, so $\phi$ indeed is a monomorphism.
Hence we have
\[C_U(v)\ug U/C\ug\Gamma(W)=\Gamma(V_1).\]
We identify $C_U(v)$ with its corresponding subgroup in
$\Gamma(V_1)$.\\

Now assume that $C_U(v)$ has a regular orbit on $V_1$.\\
Then let $v_1^*\in V$ be in a regular orbit of $C_U(v)$ on $V_1$. Then for
$w^*=v_1^*+v\in V$ we have $C_U(w^*)=C_U(v_1^*)\cap C_U(v)=C_{C_U(v)}(v_1^*)=1$, so $w^*$ is in a regular orbit of $U$ on $V$
contradicting our choice of $U$ as a counterexample.\\

Thus we now may assume that $C_U(v)$ does not have a regular orbit on
$V_1$.\\
Then from Corollary \ref{cor1} we know that there is a prime $d$ such that every element in $V_1$ is fixed by some element of order $d$ in
$C_U(v)$, and since $v_1\in V_1$ is in a $s'$-regular orbit of $C_U(v)$, clearly $d=s$.\\

Now write $|V_1|=q^n$ and recall that $C_U(v)\ug\Gamma(V_1)$. Thus by Proposition ~\ref{prop1}(d) we may assume that there exists 
a $\sigma\in\Gal(\GF(q^n):\GF(q))$ of order $s$ with $\sigma\in C_U(v)$. Next let
$N_0=\{x\in\GF(q^n)-\{0\}\ |\ \prod\limits_{\gamma\in\langle\sigma\rangle}\gamma(x)=1\}$
and $H_0=N_0\langle\sigma\rangle\leq\Gamma(V_1)$. As $C_U(v)$ does not have a $s$-regular orbit on $V_1$, by the proof of Proposition ~\ref{prop2} 
we conclude that $H_0\ug C_U(v)$ and $N_0\leq C_U(v)\cap\Gamma_0(q^n)$ (if we
identify $C_U(v)$ with its image in $\Gamma(q^n)$).\\

Clearly \[N\ug\X_{i=1}^m N/C_N(V_i)\ug\X_{i=1}^m\Gamma(V_i)\cong\Gamma(V_1)^m.\]

We now claim that\\

($*$) if $S$ is a Hall $s'$-subgroup of $N_0^m$, then
\[S\leq N\cap\Gamma_0(V_1)^m.\]

Clearly $N_0^m\leq\Gamma_0(V_1)^m$, so we only need to show $S\leq N$.
To do this, we will prove the following:\\

($**$) Let $r\not=s$ be a prime dividing $|N_0|$. Then there exists an element $y\in N_0\leq C_U(v)$
of order $|N_0|_r$ which acts fixed point freely on $V_1$ and trivially on $V_i$ for
$i=2,\ldots,m$. (Here $|N_0|_r$ denotes the $r$-part of $|N_0|$.)\\

For the moment, assume that ($**$) is already proven and let $y$ be as in
($**$).\\
Then for $i=1,\ldots,m$ let $g_i\in G$ such that $V_1^{g_i}=V_i$, and let $c_i=y^{g_i}$. Then $c_i\in N\cap\Gamma_0(V_1)^m$,
and $c_i$ acts fixed point freely on $V_i$ and trivially on all $V_j$ with $j\not=i$. Thus the group
$\langle c_1,\ldots,c_m\rangle $ is a homocyclic group of order $(|N_0|_r)^m$ and thus is a Sylow $r$-subgroup of $N_0^m$. Since $r$
is an arbitrary prime as in ($**$), we conclude that $N\cap\Gamma_0(V_1)^m$ contains a Sylow $r$-subgroup of $N_0^m$
for every such prime $r$. This implies $S\leq N\cap\Gamma_0(V_1)^m$, as stated in claim ($*$).\\

Thus to finish the proof of ($*$) we next need to prove ($**$).\\
So let $r\not=s$ be a prime dividing $|N_0|$ and let $x\in N_0\leq C_U(v)$ be an element of order $|N_0|_r$. From Proposition ~\ref{prop1} c) we know that $r>s$ and we know that $\sigma$ acts fixed point freely on $\langle x\rangle$. Consider $y=[\sigma,x]$. Then it is easy to see that $\langle y\rangle=\langle x\rangle$ and in particular $y$ is of order $|N_0|_r$. Observe that $N\ug\X_{i=1}^mN/C_N(V_i)=:M$, so let $(\sigma_1,\ldots,\sigma_m)$ with $\sigma_i\in N/C_N(V_i)$ be the image of $\sigma$ in $M$, and let $(x_1,\ldots,x_m)$ with $x_i\in N/C_N(V_i)$ be the image of $x$ in $M$. Thus $(y_1,\ldots,y_m)$ with $y_i=[\sigma_i,x_i]$ is the image of $y$ in $M$. As $\sigma\in C_N(v)=\bigcap\limits_{i=2}^mC_N(v_i)$ and also $x\in C_N(v)$, we see for $i=2,\ldots,m$ that $\sigma_i$ and $x_i$ are elements of $C_{N/C_N(V_i)}(v_i)$ which is an abelian group. Hence $(y_1,\ldots,y_m)=(y_1,1,\ldots,1)$ which means that $y$ acts trivially on $V_2,\ldots,V_m$. But clearly $y_1\not=1$ acts fixed point freely on $V_1$, so $y$ acts fixed point freely on $V_1$. Hence $y$ is an element of the kind claimed to exist in ($**$).\\

This concludes the proof of ($**$) and thus of ($*$).\\

Now suppose that $s$ does not divide $|N_0|$. Then by ($*$) we have $N_0^m\ug N\cap\Gamma_0(V_1)^m$.
Moreover recall from above that $\sigma\in C_U(v)$ is an element of order $s$ fixing each $v_i\in V_i$
($i=1,\ldots,m$); in particular $\sigma\not\in N_0^m$. \\
Now for arbitrary $x_i\in V_i$ ($i=1,\ldots,m$) it follows that there is an element of order $s$ in $N$ which centralizes
$x=\sum\limits_{i=1}^mx_i$. To see this, suppose that for some $j\in\{0,\ldots,m-1\}$ there is already an element $y_j\in N$ of order $s$
centralizing $\sum\limits_{i=1}^jx_i$ such that $y_j$ has nontrivial fixed points on each $V_i$ ($i=1,\ldots,m$).
(If $j=0$, we may choose $y_0=\sigma$.)
If $x_{j+1}=0$ or if $y_j$ acts trivially on $V_{j+1}$, then $y_j$ also
centralizes $\sum\limits_{i=1}^{j+1}x_i$; however if $y_j$ acts nontrivially on $V_{j+1}$, then as $N_0^m\ug N$, we have that
\[\langle y_j\rangle N_0\ug
C_N(\sum_{i=1}^jx_i)/(C_N(\sum_{i=1}^jx_i)\cap C_N(V_{j+1})),\]
and then we know from Proposition ~\ref{prop2} and Corollary ~\ref{cor1} 
that every element in $V_{j+1}$, so in particular $x_{j+1}$, is fixed
by some element $y_{j+1}$ of order $s$ in $\langle y_j\rangle N_0$ and since we can get $y_{j+1}$ by multiplying $y_j$ by an element of the
$(j+1)$st copy of $N_0$ in $(N_0)^m$ (i.e., by multiplying $y_j$ by an element acting trivially on all $V_k$ with $k\not=j+1$),
$y_{j+1}$ will also retain the property that it has nontrivial fixed points on each $V_k$ ($k=1,\ldots,m$).\\

As the $x_i\in V_i$ ($i=1,\ldots,m$) can be chosen arbitrarily in the previous argument, this shows that $N$ does not have a $s$-regular
orbit on $V$, contradicting our hypothesis.\\

Hence we now may assume for the remainder of the proof that $s||N_0|$, and furthermore from ($*$)
we know that $S\leq N\cap\Gamma_0(V_1)^m$ for a Hall $s'$-subgroup $S$ of $N_0^m$. Clearly $S\unlhd
N$.\\

Next we collect some properties of the group $\overline{N}=N/C_N(V_1)\leq\Gamma(V_1)$.\\

(1) $\overline{N}$ contains elementary abelian subgroups of order $s^2$. Moreover, $\overline{N}\leq\Gamma(V_1)$ contains every elementary abelian subgroup of order $s^2$ of $\Gamma(V_1)$, and they are all conjugate by elements of a Hall $s'$-subgroup of $N_0$.\\

The first statement in (1) is obvious.
To see the second part of (1), observe that all elementary abelian subgroups of order $s^2$ in $\Gamma(V_1)$ are
conjugate, since a Sylow $s$-subgroup $R$ of $\Gamma(V_1)$ has exactly one such subgroup, namely $\Omega_1(R)$
(cf. \cite[Theorem 5.4.3 and 5.4.4]{gorenstein}), and if $\sigma\in\Gal(V_1)$ (by which we mean the Galois action part of $\Gamma(V_1)$)
is of order $s$ and $x\in\Gamma_0(V_1)$
is of order $s$, then $\langle x\rangle\unlhd\Gamma(V_1)$ and $B:=\langle\sigma,x\rangle$ is elementary abelian
of order $s^2$. Now one gets all conjugates of $B$ in $\Gamma(V_1)$ by conjugation with elements in $\Gamma(V_1)$ that
$\sigma$ does not commute with. Thus the number of conjugates of $B$ is easily seen to be $|S_0|$ where $S_0$ is a
Hall $s'$-subgroup of $N_0$, as $\sigma$ acts fixed point freely on $S_0$ and trivially on the complement $S_1$ of
$S_0$ in $\Gamma_0(V_1)$ which is of order $q^\frac{n}{s}-1$.\\

We remark here that by Proposition ~\ref{prop1}(a) 
we know that $|N_0|=\frac{q^n-1}{q^\frac{n}{s}-1}$.\\
Also, since in general, $\gcd\left(\frac{a^l-1}{a-1},a-1\right)=\gcd(l,a-1)$ for $a,l\in\n$,
we see that
\[\gcd\left(\frac{q^n-1}{q^\frac{n}{s}-1},q^\frac{n}{s}-1\right)=\gcd(s,q^\frac{n}{s}-1)=s,\]
and therefore $N_0\cap S_1=\langle x\rangle$.

At any rate, since the $B^{s_0}$ for $s_0 \in S_0$ are the $|S_0|$ conjugates of $B$ in $\Gamma(V_1)$ and since
$S_0\ug\overline{N}$ in our special situation, we see that the second statement in (1) is indeed
true.\\

(2) Let $Q\in\Syl_s(\overline{N})$ and $Q_0=Q\cap\Gamma_0(V_1)$. Write $\overline{U}=U/U\cap C_N(V_1)$.
Let $r\not=s$ be a prime. Then any $r$-element of $\overline{U}$ centralizes
$Q_0$.\\

To prove (2), assume that for some prime $r\not=s$, there is a $g\in U$ such that $\overline{g}=gC_N(V_1)\in\overline{N}$
is an $r$-element not centralizing $Q_0$. Clearly we may assume that $g$ is an $r$-element and $g\not\in\Gamma_0(V_1)$.
In particular, $r$ divides $|U/(U\cap\Gamma_0(V_1)^m)|$ and thus, by the minimal choice of $s$, we have $r>s$.
As $\overline{g}$ normalizes, but does not centralize $Q_0$, by \cite[Theorem 5.2.4]{gorenstein} $\overline{g}$
acts nontrivially on $\Omega_1(Q_0)$ and thus $r$ divides $s-1$. Thus altogether we have $r\leq s-1<r-1$, which is a contradiction.
This proves (2).\\

Fix $i\in\{1,\ldots,m\}$ and let $y\in\Gamma_0(V_i)$ such that $\langle y\rangle\in\Syl_q(\Gamma_0(V_i))$.
Note that $y$ need not be in $N$. Let $y_0\in\langle y\rangle$ be of order $s$, so that $y_0\in N/C_N(V_i)$ with
the usual identifications. Moreover, let $x_0\in N/C_N(V_i)$ be of order $s$ such that $\langle x_0,y_0\rangle\leq
N/C_N(V_i)$ is elementary abelian of order $s^2$. Now as $s$ is odd, from \cite[Theorem 5.4.4]{gorenstein} we know the
exact structure of $\langle x_0,y\rangle$. In particular, $\langle x_0,y_0\rangle$ has exactly $s+1$ subgroups of order $s$,
namely $\langle y_0\rangle$ and $\langle y_0^jx_0\rangle$ for $j=0,\ldots,s-1$, and while the first one is centralized
by $y$, the remaining ones form a single orbit under conjugation by $y$, as $x_0^y=x_0y_0$. Now let $t_j\in V_j$
for $j=1,\ldots,m$ and put $t=t_1+\ldots+t_m\in V$ and $t^*=t_1+\ldots+t_{i-1}+t_i^y+t_{i+1}+\ldots+t_m$.
Furthermore, let $M$ and $M^*$ be the sets of $s'$-elements of $U$ which centralize $t$ and $t^*$, respectively. We now
claim:\\

(3) $M=M^*$.\\

To see this, first for convenience suppose that $i=1$. (This is possible since the groups $N/C_N(V_j)$ ($j=1,\ldots,m$)
are mutually isomorphic and thus the corresponding versions of (1), (2) hold for all $N/C_N(V_j)$ ($j=1,\ldots,m$)
as well.)
Now let $z_1\in M$, so $z_1\in\bigcap\limits_{i=1}^mC_U(t_i)$. To see that $z_1\in M^*$, obviously it suffices to show
that $\overline{z_1}=z_1C_N(V_1)$ centralizes $t_1^y\in V_1$ in the action of $\overline{N}$ on $V_1$.
Now $\overline{z_1}$ is a $s'$-element in $\overline{N}$ and thus by (2), centralizes $y_0\in\overline{N}$. Now view, in the
usual fashion, $\overline{z_1}$ and $y_0$ as elements in $\Gamma(V_1)$, then still $\overline{z_1}$ centralizes
$y_0\in\Gamma_0(V_1)$ and thus by \cite[Theorem 5.2.4]{gorenstein} $\overline{z_1}$ centralizes $y$. Thus
$\left(t_1^y\right)^{\overline{z_1}}=t_1^{\overline{z_1}y}=t_1^y$ (as $\overline{z_1}$ centralizes $t_1$),
so $\overline{z_1}$ indeed centralizes $t_1^y$, as wanted. This proves that $M\subseteq M^*$. Similarly, if $z_2\in M^*$,
then $\overline{z_2}=z_2C_N(V_1)$ will centralize $y$ and hence $\left(t_1^{\overline{z_2}}\right)^y=\left(t_1^y\right)^{\overline{z_2}}=
t_1^y$ implies $t_1^{\overline{z_2}}=t_1$ and so $z_2\in M$. Thus $M^*\subseteq M$, and altogether (3) is
established.\\

We now start working towards a final contradiction. Recall that $w=\sum\limits_{i=1}^mv_i$ lies in a regular orbit of $U_0$ on $V$, and
$U_0\lhd U$ with $|U/U_0|=s$, and as $w$ does not lie in a regular orbit of $U$, we have $|C_U(w)|=s$. By hypothesis $N$ has a
$s$-regular orbit on $V$, so let $a=a_1+\ldots+a_m\in V$ with $a_i\in V_i$ ($i=1,\ldots,m$) be a representative of such an orbit.
Clearly we may assume that $a_i\not=0$ for all $i=1,\ldots,m$. We claim the
following:\\

(4) For each $i\in\{1,\ldots,m\}$ let $C_U(V_i)\leq R_i\leq U$ be such that $R_i/C_U(V_i)=\Omega_1(R_i^*)$, where $R_i^*$
is the Sylow $s$-subgroup of the cyclic group $C_{U/C_U(V_i)}(v_i)$. Note that we can view
$U/C_U(V_i)\cong U/U\cap C_N(V_i)\cong UC_N(V_i)/C_N(V_i)$ to be a
subgroup of $N/C_N(V_i)$. Next let $C_N(V_i)\leq S_i\leq N$ be such that $S_i/C_N(V_i)$ is elementary abelian of order $s^2$
and such that $(S_i/C_N(V_i))\cap (C_N(a_i)/C_N(V_i))>1$ (this is possible by (1) as $C_N(a_i)/C_N(V_i)$ contains elements of order $s$).
Then there exists an $\alpha\in N$ such that $R_i\leq S_i^\alpha$ for $i=1,\ldots,m$.
In particular, by replacing $a$ by $a^\alpha$, we may assume that $R_i\leq S_i$ for $i=1,\ldots
m$.\\

We prove (4) by first observing that if $R_i/C_U(V_i)\cong R_iC_N(V_i)/C_N(V_i)$ is the trivial group for some $i$, then
for $n_i=1$ we have $R_i\leq S_i^{n_i}$. Now suppose that $R_i/C_U(V_i)$ is not trivial, i.e., $R_i/C_U(V_i)$ is cyclic of order $s$.
Then as the $N/C_N(V_j)$ ($j=1,\ldots,m$) are mutually isomorphic, from (1) we conclude that there is a $s'$-element $n_iC_N(V_i)$
(for some $n_i\in N$) of the unique subgroup of order $\left(\frac{q^n-1}{q^\frac{n}{s}-1}\right)_{s'}$ in $N/C_N(V_i)\cap\Gamma_0(V_i)$
such that
\[(R_i/C_U(V_i))\cong R_iC_N(V_i)/C_N(V_i)\leq(S_i/C_N(V_i))^{n_iC_N(V_i)}=S_i^{n_i}/C_N(V_i)\]
and thus with suitable
identification $R_i\leq S_i^{n_i}$. Note that by ($*$) we may assume that $n_i\in N_0^m$ and that $n_i$ acts trivially on $V_j$ for
all $j\not=i$. In particular, for $j\not=i$ we have that $n_i\in C_N(V_j)$ and hence $S_j^{n_i}=S_j$.
Thus if we put $\alpha=n_1n_2\ldots n_m$, then we have $\alpha\in N$ and $R_i\leq S_i^{n_i}=S_i^\alpha$, as desired. The rest of (4) now
follows immediately, so (4) is proved.\\

Next we recall that $C_U(v)$ does not have a regular orbit on $V_1$ and that
\[C_U(v)\ug U/C_U(V_1)\ug N/C_N(V_1),\]
and as the
$N/C_N(V_i)$ are all isomorphic, we see that $N/C_N(V_i)$ does not have a regular orbit on $V_i$ ($i=1,\ldots,m$). Consequently
there is a unique $Q_i\leq S_i$ such that $C_N(V_i)\leq Q_i$, $|Q_i/C_N(V_i)|=s$ and $Q_i/C_N(V_i)\leq C_{N/C_N(V_i)}(a_i)$
($i=1,\ldots,m$). Thus also $Q_i\leq C_N(a_i)$ for
$i=1,\ldots,m$.\\

Fix $i\in\{1,\ldots,m\}$. By (4) we have $R_i\leq S_i$ and $R_i/C_U(V_i)\ug S_i/C_N(V_i)\leq\Gamma(V_i)$. Now let $y_i\in\Gamma(V_i)$
such that $\langle y_i\rangle\in\Syl_s(\Gamma_0(V_i))$. From the work preceding (3) we conclude that there exists an $s_i\in\langle y_i\rangle$
such that $(R_i/C_U(V_i))^{s_i}\leq Q_i/C_N(V_i)$.
We remark here that if $R_i/C_U(V_1)=1$, then we may choose $s_i$ arbitrarily , say $s_i=1$, whereas if $R_i/C_U(V_1)$
is cyclic of order $s$, then $s_i$ is uniquely determined. We also point out that $s_i$ need not be in $N/C_N(V_i)$ and hence
$(R_i/C_U(V_i))^{s_i}$ need not be a subgroup of $U/C_U(V_i)$; in case it is not we have $C_{U/C_U(V_i)}(v_i^{s_i})=1$. So in any case
we have
\[C_{U/C_U(V_i)}(v_i^{s_i})\leq Q_i/C_N(V_i)\qquad (5).\]

We now consider $v^*=v_1^{s_1}+\ldots+v_m^{s_m}$. By (3), applied several times, we know that the set of $s'$-elements of $U$ centralizing
$v^*$ is the same as the set of $s'$-elements of $U$ centralizing $v$. As $|C_U(v)|=s$, the latter set is empty, and so we conclude that
$C_U(v^*)$ is an $s$-group. Now let $z\in C_U(v^*)$ be of order $s$. Then $z_i=zC_U(V_i)$ centralizes $v_i^{s_i}$, whence by (5)
we have $zC_U(V_i)\subseteq Q_i$ and hence $z\in Q_i$ ($i=1,\ldots,m$).
Recall that $Q_i\leq C_N(a_i)$ for $i=1,\ldots,m$. This implies that
\[z\in\bigcap_{i=1}^mC_N(a_i)=C_N(a),\]
so $s$ divides $|C_N(a)|$.
On the other hand, $a\in V$ was chosen to be in a $s$-regular orbit of $N$ on $V$. This contradiction shows that $z$ cannot exist,
and thus $C_U(v^*)=1$. So $v^*$ is in a regular orbit of $U$ on $V$. This final contradiction completes the proof of the theorem.
\end{proof}

The following problems seem to be of interest in the context of the theorem just proved.\\

\begin{rem}\label{kellerremark}

(a) Is the theorem true without the action "on top"?\\
That is, if $G\leq\Gamma(V_1)\times\ldots\times\Gamma(V_m)$ is of odd order and the $|V_i|$ are odd for all $i$, does $G$ have
a regular orbit on $V=V_1\oplus\ldots\oplus V_m$ if $G$ has a $p$-regular orbit on $V$ for every prime $p$? For $m=1$ the answer is yes by
Corollary \ref{cor1}, but in general it is open. We have not pursued this further here since it was not needed to get our main results, but it would
certainly be interesting to settle this problem.\\

(b) Can the oddness hypothesis in the theorem be weakened? Or, for that matter, can it be weakened in the more general problem stated in
(a)?\\
Again, in view of the main goal of this paper and the corresponding counterexamples we did not need to tackle these questions, but they
are natural to ask here.
\end{rem}

We are finally ready to prove our main result.

\begin{theorem}
Assume that a finite solvable group $G$ of odd order acts faithfully and irreducibly on a vector space $V$ over a field of odd characteristic. If $G$ has a $p$-regular orbit for every prime $p$ dividing $|G|$, then $G$ will have a regular orbit on $V$.
\end{theorem}
\begin{proof}
First assume $V$ is a primitive $G$-module. By ~\cite[Lemma 2.1]{AE1}, either $G$ has two regular orbits or $G \leq \Gamma(V)$. Assume $G \leq \Gamma(V)$, $G$ will have a regular orbit on $V$ by Corollary ~\ref{cor1}. Thus we may assume the action of $G$ is imprimitive.

We may assume that $V$ is induced from $V_1$. Let $H_1=\bN_G(V_1)/\bC_G(V_1)$ and we may also assume the action of $H_1$ on $V_1$ is primitive. If $H_1 \not \leq \Gamma(V_1)$ then $H_1$ has 2 regular orbits on $V_1$ by ~\cite[Lemma 2.1]{AE1}. Since $G$ is isomorphic to a subgroup of $H_1 \wr S$ for some permutation group $S$ of odd order, then with Lemma ~\ref{lem1} it is easy to see that $G$ has at least $2$ regular orbits on $V$ and we may assume $H_1 \leq \Gamma(V_1)$. Let $H=\bN_G(V_1)$, $N=\bigcap_{g \in G} H^g$ and $S=G/N$. By Theorem ~\ref{kellertheorem} we know that $N$ has at least one regular orbit on $V$. Let $z=z_1+z_2+\dots + z_m \in V$ be a representative of such an orbit, $\bC_N(z)=1$ and we may assume that all $z_i \neq 0$.

$S$ is a solvable permutation group on $\Omega=\{V_1,\dots,V_m\}$. Since $|S|$ is odd, by Lemma ~\ref{lem1}, $\Omega$ can be written as a disjoint union of $A_1,A_2$ $(A_i \neq \emptyset$, $i=1,2)$ such that $\Stab_S(A_1)=1$. Since $|H_1||V_1|$ is odd, for any element $0 \neq v_1 \in V_1$ in an orbit of $H_1$, $-v_1$ is an element in a different orbit of $H_1$ and $|v_1^{H_1}|=|(-v_1)^{H_1}|$. Thus all nontrivial orbits of $H_1$ on $V_1$ are paired and we may assign each nontrivial orbit of $H_1$ on $V_1$ a $+/-$ sign. Let $1 \leq i \leq m$ and suppose that $V_i=V_1 g$ for some $g \in G$. If $v_1$ is in a positive orbit of $H_1$ on $V_1$, then we define $v_i=v_1 g$ to be in a positive orbit of $H_i$ on $V_i$, if $v_1$ is in a negative orbit of $H_1$ on $V_1$, then we define $v_i=v_1 g$ to be in a negative orbit of $H_i$ on $V_i$. Suppose that $V_i \in A_1$; if $z_i$ is in a positive orbit of $V_i$, then we set $y_i=x_i$ and if $z_i$ is in a negative orbit of $V_i$, then we set $y_i=-x_i$. Suppose that $V_i \in A_2$; if $z_i$ is in a positive orbit of $V_i$, then we set $y_i=-x_i$ and if $z_i$ is in a negative orbit of $V_i$, then we set $y_i=x_i$. We define a new element $y=y_1+y_2+\dots + y_m$. Since $\bC_N(z_i)=\bC_N(-z_i)$, \[\bC_N(y)=\bigcap_{i=1}^{n} \bC_N(z_i) =1.\] Thus $\bC_G(y)=\bC_N(y)=1$.
\end{proof}

\section{Appendix} \label{sec:appendix}
We shall mention that Professor Thomas Wolf was aware of the imprimitive counterexample to Zhang's question many years ago. We appreciate that he allows us to include the example here.

EXAMPLE.  Let $H$ be a cyclic group of order $q^n-1$ for a prime power $q^n > 2$ and suppose that $\gcd(q^n-1, m) = 1$ for an integer $m > 1$.  Then $G = H \wr Z_m$ acts faithfully and irreducibly on a vector space V of order $q^{nm}$ with no regular orbits on $V$, but such that for each prime $p \mid |G|$, $G$ does have a $p$-regular orbit.

Proof.  Now $H$ acts transitively on the non-zero vectors of a vector space $W$ of order $q^n$.  Then $G = H \wr Z_m$ acts faithfully and irreducibly on a vector space $V = W + \cdots + W$($m$ times). Then $G$ and the base group $H\times \cdots \times H$ acts transitively on the set $C = \{(x_1,\cdots,x_m) \in V$ all $x_i$ are non-zero $\}$.  Now $C$ is a $G$-orbit of size $(q^n-1)m = |G|/m$.  Since $\gcd(q^n-1, m) = 1$ and $m > 1$, $C$ is a $p$-regular orbit for all prime divisors of  $q^n-1$, but is not a regular orbit for $G$.  If $v$ is in $V \backslash C$, then $v$ has at least one zero component and $\bC_G(v)$ has a subgroup isomorphic to $H$.  Thus $v$ is not in a regular orbit and $G$ has no regular orbits in $V$.  The $G$-orbit $D = \{(x_1,\cdots,x_m) \in V|$ exactly one $x_i$ is non-zero $\}$ has $m(q^n-1)$ elements and is a  $p$-regular orbit for all primes $p$ dividing $m$.  Thus, for every prime $p$, $G$ has a $p$-regular orbit.





\end{document}